\documentclass[11pt,letterpaper]{amsart}

\usepackage[T1]{fontenc}
\usepackage[utf8]{inputenc}
\usepackage{lmodern}
\usepackage{microtype}
\usepackage{amsmath,amsthm,amssymb,mathtools}
\usepackage[hidelinks]{hyperref}
\usepackage{enumitem}

\numberwithin{equation}{section}

\theoremstyle{plain}
\newtheorem{theorem}{Theorem}
\newtheorem{lemma}[theorem]{Lemma}
\newtheorem{proposition}[theorem]{Proposition}
\newtheorem{corollary}[theorem]{Corollary}

\theoremstyle{remark}
\newtheorem{remark}[theorem]{Remark}

\newcommand{\doi}[1]{\href{https://doi.org/#1}{doi:#1}}

\newcommand{\Z}{\mathbb{Z}}
\newcommand{\Q}{\mathbb{Q}}
\newcommand{\R}{\mathbb{R}}

\title[Legendre compressions and integrality]{Legendre compressions and an integrality conjecture for the H\"ormander--Bernhardsson extremal function}
\author{Khai-Hoan Nguyen-Dang}
\address{Morningside Center of Mathematics, Chinese Academy of Sciences, No.\ 55, Zhongguancun East Road, Beijing 100190, China}
\email{khaihoann@gmail.com}

\date{}

\begin{document}

\begin{abstract}
We prove Conjecture~2 of Bondarenko, Ortega-Cerd\`a, Radchenko, and Seip for the
three-term recurrence attached to the H\"ormander--Bernhardsson extremal function
\(\varphi\). More precisely, define
\[
\widetilde u_{-1}=0,\qquad \widetilde u_0=1,
\]
and
\[
\widetilde u_{n+1}
=
\frac{4n+2}{n+1}\bigl(n(n+1)-\lambda\bigr)\widetilde u_n
+
\frac{4n}{n+1}x\,\widetilde u_{n-1}.
\]
Then
\[
\widetilde u_n(x,\lambda)\in\Z[x,\lambda]
\qquad(n\ge0).
\]

The proof is a determinant comparison in the scaled Legendre basis. After sign reversal and
central-binomial normalization, the recurrence becomes exactly the continuant recurrence of a
finite tridiagonal compression. In particular, if \(T_n(a,\lambda)\) denotes the \(n\)th BOCRS
tridiagonal truncation, then
\[
\widetilde u_{n+1}(a^2,\lambda)=\binom{2n+2}{n+1}\det T_n(a,\lambda).
\]

As consequences, we derive that
\[
\left(\frac{\pi}{4C}\right)^2
\quad\text{and}\quad
-\frac{L_\tau(1)}{2C}
\]
are not simultaneously rational, where \(C\) is the sharp point-evaluation constant for
\(PW^1\), \(\pm\tau_n\) are the nonzero zeros of \(\varphi\), and \( L_\tau(1)=\sum_{n\ge1}\frac{(-1)^n}{\tau_n}.\)
Finally, if we write \(\varphi(z)=\sum_{n\ge0}c_n z^{2n},\) then
\[
c_n\in C^n\,\Z[\pi^2,C,L_\tau(1)]
\qquad(n\ge0).
\]
\end{abstract}

\maketitle

\section{Introduction}
\label{sec:introduction}

The H\"ormander--Bernhardsson extremal function \(\varphi\) is the unique entire function of
exponential type at most \(\pi\) such that \(\varphi(0)=1\) and
\(\|\varphi\|_{L^1(\R)}\) is minimal. Introduced by H\"ormander and Bernhardsson
\cite{HB93}, it is the extremizer in the \(L^1\) point-evaluation problem for the
Paley--Wiener space \(PW^1\); see also
\cite{BhatiaDavisKoosis89,BrevigChirreOrtegaSeip24} for related extremal questions in
Paley--Wiener and Fourier analysis.

Bondarenko, Ortega-Cerd\`a, Radchenko, and Seip (BOCRS) developed a detailed analytic theory of
\(\varphi\). They showed that
\[
\varphi(z)=\Phi(z)\Phi(-z),
\]
identified the differential equation satisfied by \(\Phi\), derived an odd inverse-power
expansion for the positive zeros \((\tau_n)_{n\ge1}\), and introduced the two-parameter
family of differential operators
\[
L_{a,b}(f)(z):=z^2f''(z)+(2z-a)f'(z)+b^2z^2f(z)
\]
governing the associated spectral problem \cite{BOCRS}. In \cite[\S10.2--\S10.3]{BOCRS} they
isolated two arithmetic conjectures. The first was proved in \cite{NguyenDangHB}. The main
purpose of the present paper is to prove the second. After the completion of the present manuscript, Danylo Radchenko informed the author that he and Wadim Zudilin have also obtained an independent proof of Conjecture~2, as part of a more general integrality theorem \cite{DZforthcoming}. Their approach is different from the determinant-comparison method used here, while the present paper also emphasizes consequences related to the irrationality question for \(L_\tau(1)\).

BOCRS observed that if \(f\) is a \(\lambda\)-eigenfunction of \(L_{1,b}\), then the even
Taylor coefficients of \(f(z)f(-z)\) satisfy the three-term recurrence
\[
u_{-1}=0,\qquad u_0=1,
\]
\[
u_{n+1}
=
\frac{4n+2}{n+1}\bigl(n(n+1)-\lambda\bigr)u_n
+
\frac{4n}{n+1}b^2u_{n-1}
\qquad(n\ge0).
\]
Conjecture~2 in \cite[\S10.3]{BOCRS} asserts that
\[
u_n(b,\lambda)\in \Z[b^2,\lambda]
\qquad(n\ge0).
\]
To work universally, we set \(x=b^2\) and define
\[
\widetilde u_{-1}=0,\qquad \widetilde u_0=1,
\]
\[
\widetilde u_{n+1}
=
\frac{4n+2}{n+1}\bigl(n(n+1)-\lambda\bigr)\widetilde u_n
+
\frac{4n}{n+1}x\,\widetilde u_{n-1}
\qquad(n\ge0).
\]
The first terms are
\[
\widetilde u_1=-2\lambda,\qquad
\widetilde u_2=6\lambda^2-12\lambda+2x,\qquad
\widetilde u_3=-20\lambda^3+160\lambda^2-12\lambda x-240\lambda+40x,
\]
already illustrating the unexpected cancellation of denominators.

Our first main result proves that the entire two-parameter family $\widetilde u_n(x,\lambda)$ is integral \emph{before} specialization.
\begin{theorem}\label{thm:main}
For every $n\ge 0$ one has
\[
 \widetilde u_n(x,\lambda)\in \Z[x,\lambda].
\]
Consequently,
\[
 u_n(b,\lambda)=\widetilde u_n(b^2,\lambda)\in \Z[b^2,\lambda]\subset \Z[b,\lambda]
\qquad(n\ge 0).
\]
\end{theorem}

The second main result identifies the same sequence with exact finite BOCRS compressions. Writing
\[
\widehat u_n(a,\lambda):=\widetilde u_n(a^2,\lambda),
\]
and letting \(T_n(a,\lambda)\) denote the tridiagonal matrix defined by 
\[
T_0(a,\lambda):=(-\lambda),
\]
and, for \(n\ge 1\),
\begin{equation}
T_n(a,\lambda):=
\begin{pmatrix}
-\lambda & \dfrac{a}{3} & 0 & \cdots & 0 \\
-a & 2-\lambda & \dfrac{2a}{5} & \ddots & \vdots \\
0 & -\dfrac{2a}{3} & 6-\lambda & \ddots & 0 \\
\vdots & \ddots & \ddots & \ddots & \dfrac{na}{2n+1} \\
0 & \cdots & 0 & -\dfrac{na}{2n-1} & n(n+1)-\lambda
\end{pmatrix}.
\end{equation}

\begin{theorem}[Exact finite-compression identity]
For every \(n\ge 0\),
\[
\widehat u_{n+1}(a,\lambda)=\binom{2n+2}{n+1}\det T_n(a,\lambda).
\]
\end{theorem}

Thus the BOCRS truncation scheme is not merely numerical: after the canonical
central-binomial scaling, it is exactly the characteristic-polynomial sequence of the
finite principal compressions.

The paper also derives two arithmetic consequences. The rational-slice theorem
(Theorem~\ref{thm:rational-slice}) shows that if
\[
F_{a,\lambda}(z):=\sum_{n\ge0}\widehat u_n(a,\lambda)z^{2n}
\]
is entire, nonpolynomial, and of exponential type at most \(2a\), then \(a^2\) and
\(\lambda\) are not both rational. Applied to the H\"ormander--Bernhardsson branch, this yields
Corollary~\ref{cor:HB-normalized-obstruction}:
\[
\left(\frac{\pi}{4C}\right)^2
\qquad\text{and}\qquad
-\frac{L_\tau(1)}{2C}
\]
are not both rational. Here \(C\) is the sharp point-evaluation constant and
\[
L_\tau(s):=\sum_{n\ge1}\frac{(-1)^n}{\tau_n^s}.
\]

A second consequence is a coefficient-ring theorem. If
\[
\varphi(z)=\sum_{n=0}^{\infty}c_n z^{2n},
\]
then Corollary~\ref{cor:HB-coefficient-ring} shows that
\[
c_n\in C^n\,\Z[\pi^2,C,L_\tau(1)]
\qquad(n\ge0).
\]
In particular, only even powers of \(\pi\) can occur in the Taylor coefficients of
\(\varphi\). This is obtained from a sharp weighted refinement of the universal integrality
statement: every monomial \(x^r\lambda^s\) occurring in \(\widetilde u_n(x,\lambda)\)
satisfies
\[
2r+s\le n.
\]

\subsection*{Method of proof}
The point of the paper is not merely that denominators cancel. The proof identifies the
mechanism that forces the cancellation, and that mechanism is conceptually unexpected in the
BOCRS setting. Integrality phenomena of this kind are often associated with regular-singular
geometry, i.e., Picard--Fuchs equations, diagonals, \(G\)-functions, or related algebraic
structures. BOCRS explicitly emphasize that Conjecture~2 lies outside that framework: the
differential equation attached to \(L_{a,b}\) has irregular singularities at \(0\) and
\(\infty\), and the conjectured integrality therefore calls for a different explanation
\cite[\S4,\S10.3]{BOCRS}. Our contribution is to show that the arithmetic structure is
governed instead by an exact finite-dimensional Legendre compression.

The proof begins by the sign reversal \(x=-\beta^2\) and the normalization
\[
v_n(x,\lambda):=\frac{\widetilde u_n(x,\lambda)}{\binom{2n}{n}}.
\]
After dividing by the central binomial coefficient, the recurrence becomes exactly the
continuant recurrence of a tridiagonal determinant arising from the action of
\[
\mathcal L+\beta t-\lambda,
\qquad
\mathcal L:=-\frac{d}{dt}\Bigl((1-t^2)\frac{d}{dt}\Bigr),
\]
on the scaled Legendre basis \(Q_n(t)=2^nP_n(t)\). The same operator, written in the monomial
basis, gives a matrix with entries in \(\Z[\beta,\lambda]\). Comparing the two determinant
realizations via change of basis yields
\[
\det M_n=-\widetilde u_n(-\beta^2,\lambda),
\]
and hence integrality. In particular, the central binomial coefficient is not an auxiliary
normalization: it is the exact change-of-basis factor between the Legendre and monomial
compressions.

Taken together with \cite{NguyenDangHB}, the two arithmetic conjectures isolated in
\cite[\S10.2--\S10.3]{BOCRS} are now resolved.

We also emphasize that the BOCRS recurrence is holonomic in the usual sense: after clearing the
denominator \(n+1\), it becomes a second-order \(P\)-recurrence, and its ordinary generating
series is \(D\)-finite. However, this formal holonomic structure plays no essential role in the
proof. The arithmetic content of Theorem~\ref{thm:main} is not a consequence of holonomicity
alone; it comes from an exact finite-dimensional Legendre compression identity, which is
particularly unexpected here because the underlying BOCRS differential equation has irregular
singularities.

\subsection*{Organization of the paper}

The paper is organized as follows. Section~\ref{sec:main-proof} proves
Theorem~\ref{thm:main} by the Legendre determinant comparison.
Section~\ref{sec:continued-fractions} develops the BOCRS truncation and continued-fraction
formalism, proves Theorem~\ref{thm:truncation-identity}, and derives the rational-slice
consequences. Section~\ref{sec:weighted-arithmetic} proves the weighted support bound and the
coefficient-ring theorem.

\subsection*{Acknowledgements}

We thank Quoc-Hung Nguyen for bringing this conjecture to our attention. We thank Danylo Radchenko for kindly informing us of his forthcoming joint paper with Wadim Zudilin. We are grateful to the Morningside Center of Mathematics, Chinese Academy of Sciences, for its support and a stimulating research environment.

\section{Legendre determinant proof of Theorem~\ref{thm:main}}
\label{sec:main-proof}

\subsection{The universal recurrence}

We work with the BOCRS recurrence
\begin{equation}\label{eq:original-recurrence}
 u_{-1}=0,\qquad u_0=1,
\end{equation}
\begin{equation}\label{eq:u-recurrence}
 u_{n+1}=\frac{4n+2}{n+1}\bigl(n(n+1)-\lambda\bigr)u_n
 +\frac{4n}{n+1}b^2u_{n-1}
 \qquad(n\ge 0).
\end{equation}
Since the parameter \(b\) appears only through \(b^2\), it is convenient to introduce a new
variable \(x\) and define
\begin{equation}\label{eq:utilde-recurrence-init}
 \widetilde u_{-1}=0,\qquad \widetilde u_0=1,
\end{equation}
\begin{equation}\label{eq:utilde-recurrence}
 \widetilde u_{n+1}=\frac{4n+2}{n+1}\bigl(n(n+1)-\lambda\bigr)\widetilde u_n
 +\frac{4n}{n+1}x\,\widetilde u_{n-1}
 \qquad(n\ge 0).
\end{equation}
By uniqueness of recursively defined sequences in \(\Q[x,\lambda]\),
\[
u_n(b,\lambda)=\widetilde u_n(b^2,\lambda)\qquad(n\ge -1).
\]
Thus Theorem~\ref{thm:main} implies Conjecture~2 from \cite[\S10.3]{BOCRS}.

The proof has four steps. First, by Lemma~\ref{lem:beta-substitution} it is enough to prove
integrality after the sign reversal \(x=-\beta^2\). Second, after dividing by the central
binomial coefficient \(C_n=\binom{2n}{n}\), the recurrence becomes exactly the continuant
recurrence of a tridiagonal determinant \(D_n\) arising from the action of
\(\mathcal L+\beta t-\lambda\) on the scaled Legendre basis \(Q_n=2^nP_n\). Third, the same
operator gives an integral matrix \(M_n\) in the monomial basis. Fourth, comparison of the
two determinant realizations through change of basis yields
\[
\det M_n=-\widetilde u_n(-\beta^2,\lambda),
\]
which forces integrality. The remaining subsections carry out these four steps.

\subsection{Reduction to a quadratic specialization}

The first step is a simple algebraic lemma.

\begin{lemma}\label{lem:beta-substitution}
Let $f(x,\lambda)\in \Q[x,\lambda]$. If
\[
 f(-\beta^2,\lambda)\in \Z[\beta,\lambda],
\]
then in fact
\[
 f(x,\lambda)\in \Z[x,\lambda].
\]
\end{lemma}

\begin{proof}
Write
\[
 f(x,\lambda)=\sum_{r=0}^m a_r(\lambda)x^r,
 \qquad a_r(\lambda)\in \Q[\lambda].
\]
Then
\[
 f(-\beta^2,\lambda)=\sum_{r=0}^m (-1)^r a_r(\lambda)\beta^{2r}.
\]
Now $\Z[\beta,\lambda]=\Z[\lambda][\beta]$ is a polynomial ring in the indeterminate $\beta$ over $\Z[\lambda]$. Hence the coefficients of a polynomial in $\beta$ are unique. Since $f(-\beta^2,\lambda)$ lies in $\Z[\lambda][\beta]$, each coefficient of $\beta^{2r}$ must belong to $\Z[\lambda]$ and each odd coefficient must vanish. Therefore
\[
 (-1)^r a_r(\lambda)\in \Z[\lambda]
 \qquad(0\le r\le m),
\]
so $a_r(\lambda)\in \Z[\lambda]$ for all $r$. Thus $f(x,\lambda)\in \Z[x,\lambda]$.
\end{proof}

By Lemma \ref{lem:beta-substitution}, Theorem \ref{thm:main} will follow once we prove that
\[
 \widetilde u_n(-\beta^2,\lambda)\in \Z[\beta,\lambda]
 \qquad(n\ge 0).
\]
We now turn to the determinantal construction that yields this fact.

\subsection{Scaled Legendre polynomials}

Let \(P_n(t)\) denote the classical Legendre polynomial of degree \(n\), normalized by the generating function
\[
\sum_{n\ge 0} P_n(t)\,z^n=\frac{1}{\sqrt{1-2tz+z^2}}.
\]
Equivalently, \(P_n\) is given by Rodrigues' formula
\[
P_n(t)=\frac{1}{2^n n!}\frac{d^n}{dt^n}(t^2-1)^n.
\]

For our purposes it is convenient to use the scaled polynomials
\[
Q_n(t):=2^nP_n(t)\qquad(n\ge 0).
\]
Then \(Q_n(t)\in \mathbb Z[t]\), and one has the explicit expansion
\[
Q_n(t)
=
\sum_{k=0}^{\lfloor n/2\rfloor}
(-1)^k
\binom{n}{k}
\binom{2n-2k}{n}
\,t^{\,n-2k}.
\]
In particular, the leading coefficient of \(Q_n\) is
\[
C_n:=\binom{2n}{n}.
\]

The first values are
\begin{align*}
Q_0(t)&=1,\\
Q_1(t)&=2t,\\
Q_2(t)&=6t^2-2,\\
Q_3(t)&=20t^3-12t,\\
Q_4(t)&=70t^4-60t^2+6,\\
Q_5(t)&=252t^5-280t^3+60t.
\end{align*}

The Legendre differential equation becomes
\[
-\frac{d}{dt}\!\left((1-t^2)\frac{d}{dt}Q_n(t)\right)=n(n+1)\,Q_n(t).
\]
Thus, if we set
\[
\mathcal L:=-\frac{d}{dt}\!\left((1-t^2)\frac{d}{dt}\right),
\]
then
\[
\mathcal L Q_n=n(n+1)Q_n.
\]

Set \(Q_{-1}:=0\) for convenience. The three-term recurrence for the scaled family is
\[
(n+1)Q_{n+1}(t)=2(2n+1)t\,Q_n(t)-4nQ_{n-1}(t)
\qquad(n\ge 1),
\]
with initial conditions
\[
Q_0(t)=1,\qquad Q_1(t)=2t.
\]
With the convention \(Q_{-1}=0\), this can be rewritten as
\[
t\,Q_n(t)
=
\frac{n+1}{2(2n+1)}\,Q_{n+1}(t)
+
\frac{2n}{2n+1}\,Q_{n-1}(t)
\qquad(n\ge 0).
\]

It follows that
\[
(\mathcal L+\beta t-\lambda)Q_n
=
\bigl(n(n+1)-\lambda\bigr)Q_n
+
\frac{2n\beta}{2n+1}Q_{n-1}
+
\frac{(n+1)\beta}{2(2n+1)}Q_{n+1}
\qquad(n\ge 0).
\]
This is the tridiagonal action used in the determinant comparison argument.

Finally, \(Q_n\) has the parity property
\[
Q_n(-t)=(-1)^nQ_n(t),
\]
so \(Q_n\) is even for \(n\) even and odd for \(n\) odd. For standard formulas for Legendre polynomials, see, for example, DLMF, Chapter~18, especially \S\S18.5, 18.8, and 18.9.

We have shown the following proposition.

\begin{proposition}\label{prop:Legendre-facts}
For every $n\ge 0$ the polynomial $Q_n$ belongs to $\Z[t]$ and has the explicit expansion
\begin{equation}\label{eq:Q-explicit}
 Q_n(t)=\sum_{k=0}^{\lfloor n/2\rfloor}
 (-1)^k\binom{n}{k}\binom{2n-2k}{n}t^{n-2k}.
\end{equation}
Also set \(Q_{-1}:=0\). In particular, the leading coefficient of $Q_n$ is
\begin{equation}\label{eq:Cn}
 C_n:=\binom{2n}{n}.
\end{equation}
Moreover, with
\begin{equation}\label{eq:L-operator}
 \mathcal L:=-\frac{d}{dt}\Bigl((1-t^2)\frac{d}{dt}\Bigr),
\end{equation}
one has
\begin{equation}\label{eq:Legendre-eigen}
 \mathcal LQ_n=n(n+1)Q_n,
\end{equation}
and the three-term recurrence
\begin{equation}\label{eq:Q-recurrence}
 (n+1)Q_{n+1}=2(2n+1)tQ_n-4nQ_{n-1}
 \qquad(n\ge 1).
\end{equation}
Equivalently,
\begin{equation}\label{eq:tQn}
 tQ_n=\frac{n+1}{2(2n+1)}Q_{n+1}+\frac{2n}{2n+1}Q_{n-1}
 \qquad(n\ge 0).
\end{equation}
Consequently,
\begin{equation}\label{eq:operator-on-Q}
 (\mathcal L+\beta t-\lambda)Q_n
 =(n(n+1)-\lambda)Q_n
 +\frac{2n\beta}{2n+1}Q_{n-1}
 +\frac{(n+1)\beta}{2(2n+1)}Q_{n+1}
 \qquad(n\ge 0).
\end{equation}
\end{proposition}

\subsection{An integral determinant}

Fix
\[
 R:=\Z[\beta,\lambda].
\]
For $m\ge 0$ we denote by $R[t]_{\le m}$ the $R$-module of polynomials in $t$ of degree at most $m$.

For every integer $n\ge 1$ define an $R$-linear map
\begin{equation}\label{eq:Phi-def}
 \Phi_n:R[t]_{\le n-1}\oplus R\longrightarrow R[t]_{\le n},
 \qquad
 \Phi_n(f,\mu):=(\mathcal L+\beta t-\lambda)f-\mu Q_n.
\end{equation}
Let $e:=(0,1)\in R[t]_{\le n-1}\oplus R$. We consider the monomial bases
\[
 \mathcal B_n^{\mathrm{dom}}=(1,t,\dots,t^{n-1},e),
 \qquad
 \mathcal B_n^{\mathrm{cod}}=(1,t,\dots,t^n).
\]
Let $M_n$ be the matrix of $\Phi_n$ with respect to these bases.

\begin{lemma}\label{lem:M-integral}
For every $n\ge 1$, the matrix $M_n$ has entries in $R$. In particular,
\[
 \det M_n\in R.
\]
\end{lemma}

\begin{proof}
For $m\ge 0$ one computes directly from \eqref{eq:L-operator} that
\begin{align*}
 (\mathcal L+\beta t-\lambda)t^m
 &= -\frac{d}{dt}\bigl((1-t^2)mt^{m-1}\bigr)+\beta t^{m+1}-\lambda t^m \\
 &= -\frac{d}{dt}\bigl(mt^{m-1}-mt^{m+1}\bigr)+\beta t^{m+1}-\lambda t^m \\
 &= -\bigl(m(m-1)t^{m-2}-m(m+1)t^m\bigr)+\beta t^{m+1}-\lambda t^m \\
 &= \beta t^{m+1}+\bigl(m(m+1)-\lambda\bigr)t^m-m(m-1)t^{m-2}.
\end{align*}
Here the term $t^{m-2}$ is absent when $m\le 1$, because then its coefficient $m(m-1)$ is zero. Since also $Q_n\in \Z[t]\subset R[t]$ by Proposition \ref{prop:Legendre-facts}, each column of $M_n$ has entries in $R$. Thus $M_n\in \operatorname{Mat}_{n+1}(R)$ and therefore $\det M_n\in R$.
\end{proof}

\subsection{The same determinant in the Legendre basis}

The matrix $M_n$ was defined over the ring $R$. To compare it with a matrix written in the $Q$-basis, we now extend scalars to the fraction field
\[
 K:=\Q(\beta,\lambda).
\]
Let
\[
 \Phi_{n,K}:=\Phi_n\otimes_R K:
 K[t]_{\le n-1}\oplus K\longrightarrow K[t]_{\le n}.
\]
Because the leading coefficient of $Q_j$ is the nonzero integer $C_j$, the families
\[
 \mathcal C_n^{\mathrm{dom}}=(Q_0,Q_1,\dots,Q_{n-1},e),
 \qquad
 \mathcal C_n^{\mathrm{cod}}=(Q_0,Q_1,\dots,Q_n)
\]
are $K$-bases.

\begin{remark}\label{rem:scalar-extension}
This scalar extension is essential. Over $R$ the family $(Q_0,Q_1)$ is \emph{not} a basis of $R[t]_{\le 1}$, since $Q_1=2t$ and $t\notin RQ_0+RQ_1$. Thus every change-of-basis argument involving the $Q$-basis must be performed only after passing to a ring in which the integers are invertible; the field $K$ is the cleanest choice.
\end{remark}

Let $N_n$ be the matrix of $\Phi_{n,K}$ with respect to the bases $\mathcal C_n^{\mathrm{dom}}$ and $\mathcal C_n^{\mathrm{cod}}$.

\begin{lemma}\label{lem:N-block}
For every $n\ge 1$, the matrix $N_n$ has the block form
\[
 N_n=
 \begin{pmatrix}
  B_n & 0 \\
  * & -1
 \end{pmatrix},
\]
where $B_n=(b_{ij})_{0\le i,j\le n-1}$ is the tridiagonal matrix determined by
\begin{align*}
 b_{jj}&=j(j+1)-\lambda &&(0\le j\le n-1),\\
 b_{j-1,j}&=\frac{2j\beta}{2j+1} &&(1\le j\le n-1),\\
 b_{j+1,j}&=\frac{(j+1)\beta}{2(2j+1)} &&(0\le j\le n-2),
\end{align*}
and all other entries equal to $0$. Equivalently,
\[
 B_n=
 \begin{pmatrix}
 -\lambda & \dfrac{2\beta}{3} & 0 & \cdots & 0 \\
 \dfrac{\beta}{2} & 2-\lambda & \dfrac{4\beta}{5} & \ddots & \vdots \\
 0 & \dfrac{\beta}{3} & 6-\lambda & \ddots & 0 \\
 \vdots & \ddots & \ddots & \ddots & \dfrac{2(n-1)\beta}{2n-1} \\
 0 & \cdots & 0 & \dfrac{(n-1)\beta}{2(2n-3)} & n(n-1)-\lambda
 \end{pmatrix}.
\]
In particular,
\begin{equation}\label{eq:detN}
 \det N_n=-\det B_n.
\end{equation}
\end{lemma}

\begin{proof}
For $0\le j\le n-1$, equation \eqref{eq:operator-on-Q} gives
\[
 \Phi_{n,K}(Q_j,0)
 =(j(j+1)-\lambda)Q_j
 +\frac{2j\beta}{2j+1}Q_{j-1}
 +\frac{(j+1)\beta}{2(2j+1)}Q_{j+1}.
\]
Thus the first $n$ columns of $N_n$ are tridiagonal in the basis $(Q_0,\dots,Q_n)$. The last basis vector $e$ is sent to
\[
 \Phi_{n,K}(0,1)=-Q_n,
\]
so the last column is zero except for a single entry $-1$ in the last row. This proves the asserted block form, and then \eqref{eq:detN} follows immediately by expansion along the last column.
\end{proof}

For $n\ge 1$ set
\[
 D_n:=\det B_n,
\]
and define also $D_0:=1$ (the determinant of the empty matrix).

\begin{lemma}\label{lem:D-recurrence}
The determinants $D_n$ satisfy
\begin{equation}\label{eq:D-recurrence}
 D_{n+1}=(n(n+1)-\lambda)D_n
 -\frac{n^2\beta^2}{(2n-1)(2n+1)}D_{n-1}
 \qquad(n\ge 1),
\end{equation}
with initial values
\begin{equation}\label{eq:D-initial}
 D_0=1,
 \qquad
 D_1=-\lambda.
\end{equation}
\end{lemma}

\begin{proof}
The value \(D_1=-\lambda\) is immediate from the \(1\times 1\) matrix \(B_1=(-\lambda)\).
Now fix \(n\ge 1\). Since \(B_{n+1}\) is tridiagonal, its last row has exactly two
nonzero entries:
\[
\frac{n\beta}{2(2n-1)}
\quad\text{in the penultimate column},
\qquad
n(n+1)-\lambda
\quad\text{in the last column}.
\]
Expanding \(\det B_{n+1}\) along that row gives
\[
D_{n+1}
=
(n(n+1)-\lambda)D_n
-
\frac{n\beta}{2(2n-1)}\,M_{n+1,n},
\]
where \(M_{n+1,n}\) is the minor obtained by deleting the last row and the
penultimate column.

In this minor, the last column has a single nonzero entry
\[
\frac{2n\beta}{2n+1}
\]
in its last row. Expanding \(M_{n+1,n}\) along that last column therefore yields
\[
M_{n+1,n}=\frac{2n\beta}{2n+1}D_{n-1}.
\]
Substituting this into the previous identity, we obtain
\[
D_{n+1}
=
(n(n+1)-\lambda)D_n
-
\frac{n\beta}{2(2n-1)}\cdot \frac{2n\beta}{2n+1}D_{n-1},
\]
which is exactly \eqref{eq:D-recurrence}.
\end{proof}

\subsection{Normalization of the recurrence}

We now normalize $\widetilde u_n$ by the central binomial coefficients
\[
 C_n:=\binom{2n}{n}
 \qquad(n\ge 0),
\]
already introduced in \eqref{eq:Cn}. Define
\begin{equation}\label{eq:vn-def}
 v_n(x,\lambda):=\frac{\widetilde u_n(x,\lambda)}{C_n}
 \qquad(n\ge 0).
\end{equation}

\begin{lemma}\label{lem:v-recurrence}
The sequence $(v_n)_{n\ge 0}$ satisfies
\begin{equation}\label{eq:v-recurrence}
 v_{n+1}=(n(n+1)-\lambda)v_n
 +\frac{n^2}{(2n-1)(2n+1)}x\,v_{n-1}
 \qquad(n\ge 1),
\end{equation}
with initial values
\begin{equation}\label{eq:v-initial}
 v_0=1,
 \qquad
 v_1=-\lambda.
\end{equation}
\end{lemma}

\begin{proof}
Since
\[
 \frac{C_{n+1}}{C_n}=\frac{2(2n+1)}{n+1},
\]
we have
\[
 \frac{C_n}{C_{n+1}}=\frac{n+1}{2(2n+1)}.
\]
Dividing \eqref{eq:utilde-recurrence} by $C_{n+1}$ yields
\begin{align*}
 v_{n+1}
 &=\frac{4n+2}{n+1}\bigl(n(n+1)-\lambda\bigr)\frac{C_n}{C_{n+1}}v_n
 +\frac{4n}{n+1}x\frac{C_{n-1}}{C_{n+1}}v_{n-1} \\
 &=\bigl(n(n+1)-\lambda\bigr)v_n
 +\frac{4n}{n+1}x\frac{C_{n-1}}{C_{n+1}}v_{n-1}.
\end{align*}
A direct calculation gives
\[
 \frac{C_{n-1}}{C_{n+1}}=
 \frac{n(n+1)}{4(2n-1)(2n+1)},
\]
so \eqref{eq:v-recurrence} follows. Finally,
\[
 v_0=\frac{\widetilde u_0}{C_0}=1,
 \qquad
 v_1=\frac{\widetilde u_1}{C_1}=\frac{-2\lambda}{2}=-\lambda.
\]
\end{proof}

\begin{corollary}\label{cor:D=v}
For every $n\ge 0$,
\begin{equation}\label{eq:D-equals-v}
 D_n=v_n(-\beta^2,\lambda)=\frac{\widetilde u_n(-\beta^2,\lambda)}{C_n}.
\end{equation}
\end{corollary}

\begin{proof}
After the specialization $x=-\beta^2$, the recurrence \eqref{eq:v-recurrence} becomes
\[
 v_{n+1}(-\beta^2,\lambda)
 =(n(n+1)-\lambda)v_n(-\beta^2,\lambda)
 -\frac{n^2\beta^2}{(2n-1)(2n+1)}v_{n-1}(-\beta^2,\lambda)
\]
for $n\ge 1$, with the initial values \eqref{eq:v-initial}. By Lemma \ref{lem:D-recurrence}, this is exactly the recurrence and initial data satisfied by $D_n$. Hence $D_n=v_n(-\beta^2,\lambda)$ for all $n\ge 0$.
\end{proof}

\subsection{Comparison of the two determinants}

We now compare the monomial-basis matrix $M_n$ and the Legendre-basis matrix $N_n$.

Let
\[
 S_n\in \operatorname{GL}_{n+1}(K)
\]
be the change-of-basis matrix from $\mathcal C_n^{\mathrm{cod}}=(Q_0,\dots,Q_n)$ to the monomial basis $\mathcal B_n^{\mathrm{cod}}=(1,t,\dots,t^n)$, and let
\[
 T_n\in \operatorname{GL}_{n+1}(K)
\]
be the change-of-basis matrix from $\mathcal C_n^{\mathrm{dom}}=(Q_0,\dots,Q_{n-1},e)$ to $\mathcal B_n^{\mathrm{dom}}=(1,t,\dots,t^{n-1},e)$.

Because \(Q_j\) has degree \(j\) and leading coefficient \(C_j\), both \(S_n\) and \(T_n\)
are triangular. Their diagonals are
\[
(C_0,\dots,C_n)
\qquad\text{and}\qquad
(C_0,\dots,C_{n-1},1),
\]
respectively. Hence
\begin{equation}\label{eq:detS}
 \det S_n=\prod_{j=0}^n C_j,
\end{equation}
and
\begin{equation}\label{eq:detT}
 \det T_n=\prod_{j=0}^{n-1} C_j.
\end{equation}
Since $M_n$ and $N_n$ represent the same $K$-linear map $\Phi_{n,K}$ in different bases, they are related by
\begin{equation}\label{eq:M-SNT}
 M_n=S_nN_nT_n^{-1}.
\end{equation}
Taking determinants and using \eqref{eq:detS}--\eqref{eq:M-SNT} yields
\begin{equation}\label{eq:detM-first}
 \det M_n=\frac{\det S_n}{\det T_n}\det N_n=C_n\det N_n.
\end{equation}
Combining \eqref{eq:detM-first} with \eqref{eq:detN} and \eqref{eq:D-equals-v}, we obtain
\begin{equation}\label{eq:detM-final}
 \det M_n=-C_nD_n=-\widetilde u_n(-\beta^2,\lambda).
\end{equation}
By Lemma \ref{lem:M-integral}, the left-hand side belongs to $R=\Z[\beta,\lambda]$. Therefore
\[
 \widetilde u_n(-\beta^2,\lambda)\in \Z[\beta,\lambda]
 \qquad(n\ge 1).
\]
The same conclusion also holds for $n=0$, because $\widetilde u_0=1$.

Applying Lemma \ref{lem:beta-substitution}, we conclude that
\[
 \widetilde u_n(x,\lambda)\in \Z[x,\lambda]
 \qquad(n\ge 0).
\]
This proves Theorem \ref{thm:main}.

\section{BOCRS truncations, continued fractions, and rational slices}
\label{sec:continued-fractions}

For later use, we introduce the specialization
\[
\widehat u_n(a,\lambda):=\widetilde u_n(a^2,\lambda)
\qquad(n\ge -1).
\]
Thus
\begin{equation}\label{eq:uhat-recurrence}
\widehat u_{-1}=0,
\qquad
\widehat u_0=1,
\qquad
\widehat u_{n+1}
=
\frac{4n+2}{n+1}\bigl(n(n+1)-\lambda\bigr)\widehat u_n
+\frac{4n}{n+1}a^2\widehat u_{n-1}
\qquad(n\ge 0),
\end{equation}
and Theorem~\ref{thm:main} gives
\begin{equation}\label{eq:uhat-integrality}
\widehat u_n(a,\lambda)\in \Z[a^2,\lambda]
\qquad(n\ge 0).
\end{equation}
Moreover, it is immediate from \eqref{eq:uhat-recurrence} that
\begin{equation}\label{eq:uhat-degree-bounds}
\deg_{a^2}\widehat u_n\le \Bigl\lfloor\frac{n}{2}\Bigr\rfloor,
\qquad
\deg_{\lambda}\widehat u_n\le n
\qquad(n\ge 0).
\end{equation}

We next recall the BOCRS three-term recurrence in the form used for the decaying branch.

\begin{proposition}\label{prop:boq}
Let $a>0$, let $\lambda\in\mathbb R$, and let $(\xi_n)_{n\ge 0}$ be a positive decaying solution of
\begin{equation}\label{eq:BOCRS-xi-recurrence}
\lambda\xi_n
=
 n(n+1)\xi_n
 -a\left(\frac{n}{2n-1}\xi_{n-1}-\frac{n+1}{2n+3}\xi_{n+1}\right),
 \qquad \xi_{-1}=0.
\end{equation}
For $n\ge 0$ set
\begin{equation}\label{eq:beta-def}
\beta_n:=\frac{a^2(n+1)^2}{(2n+1)(2n+3)},
\end{equation}
and for $n\ge 1$ define
\begin{equation}\label{eq:q-def}
q_n:=\frac{an}{2n-1}\frac{\xi_{n-1}}{\xi_n}.
\end{equation}
Then
\begin{equation}\label{eq:q-riccati}
q_n=n(n+1)-\lambda+\frac{\beta_n}{q_{n+1}}
\qquad(n\ge 1).
\end{equation}
In particular, $q_n\to\infty$, and
\begin{equation}\label{eq:q1-cf}
q_1
=
2-\lambda+\cfrac{\beta_1}{6-\lambda+\cfrac{\beta_2}{12-\lambda+\cfrac{\beta_3}{20-\lambda+\ddots}}}.
\end{equation}
If, in addition, we normalize by \(\xi_0=1\), then \(\lambda>0\) and
\begin{equation}\label{eq:q1-lambda}
q_1=\frac{a^2}{3\lambda}.
\end{equation}
Hence
\begin{equation}\label{eq:lambda-cf}
\lambda=
\cfrac{a^2/3}{2-\lambda+\cfrac{4a^2/15}{6-\lambda+\cfrac{9a^2/35}{12-\lambda+\ddots}}}.
\end{equation}
\end{proposition}

\begin{proof}
Divide \eqref{eq:BOCRS-xi-recurrence} by $\xi_n$ to get
\[
\lambda
=
n(n+1)-q_n+\frac{a(n+1)}{2n+3}\frac{\xi_{n+1}}{\xi_n}.
\]
By the definition of $q_{n+1}$,
\[
\frac{a(n+1)}{2n+3}\frac{\xi_{n+1}}{\xi_n}
=
\frac{a^2(n+1)^2}{(2n+1)(2n+3)}\frac{1}{q_{n+1}}
=
\frac{\beta_n}{q_{n+1}},
\]
which gives \eqref{eq:q-riccati}. Since $q_{n+1}>0$, we have
\[
q_n\ge n(n+1)-\lambda,
\]
so $q_n\to\infty$. Iterating \eqref{eq:q-riccati} therefore gives the continued fraction \eqref{eq:q1-cf}. Finally, the equation \eqref{eq:BOCRS-xi-recurrence} at \(n=0\) reads
\[
\lambda\xi_0=\frac{a}{3}\xi_1.
\]
Since \(a>0\) and \(\xi_0,\xi_1>0\), this shows that \(\lambda>0\).
With the normalization \(\xi_0=1\), we therefore get
\[
q_1=a\frac{\xi_0}{\xi_1}=\frac{a^2}{3\lambda},
\]
which is \eqref{eq:q1-lambda}; substituting this into \eqref{eq:q1-cf} gives
\eqref{eq:lambda-cf}.
\end{proof}

\begin{remark}\label{rem:HB-cf-relation}
For the H\"ormander--Bernhardsson branch, the rescaled eigenfunction
\[
h(z):=\Phi\!\left(\frac{z}{2C}\right)
\]
corresponds to the BOCRS decaying branch for \(L_{1,a}\) at the normalized spectral pair
\[
a_*=\frac{\pi}{4C},
\qquad
\lambda_*=-\frac{L_\tau(1)}{2C}.
\]
Hence Proposition~\ref{prop:boq} applies, and substituting
\((a,\lambda)=(a_*,\lambda_*)\) into \eqref{eq:lambda-cf} gives
\begin{equation}\label{eq:HB-continued-fraction}
-\frac{L_\tau(1)}{2C}
=
\cfrac{\pi^2/(48C^2)}
{\,2+\frac{L_\tau(1)}{2C}
+\cfrac{\pi^2/(60C^2)}
{\,6+\frac{L_\tau(1)}{2C}
+\cfrac{9\pi^2/(560C^2)}
{\,12+\frac{L_\tau(1)}{2C}+\ddots}}}.
\end{equation}
Equivalently,
\begin{equation}\label{eq:HB-continued-fraction-Ltau}
L_\tau(1)
=
-\cfrac{\pi^2/(24C)}
{\,2+\frac{L_\tau(1)}{2C}
+\cfrac{\pi^2/(60C^2)}
{\,6+\frac{L_\tau(1)}{2C}
+\cfrac{9\pi^2/(560C^2)}
{\,12+\frac{L_\tau(1)}{2C}+\ddots}}}.
\end{equation}

It is sometimes convenient to rewrite this in terms of the normalized quantities
\[
X:=\frac{C}{\pi},
\qquad
Y:=\frac{L_\tau(1)}{C}.
\]
Then \eqref{eq:HB-continued-fraction} becomes
\[
-\frac{Y}{2}
=
\cfrac{1/(48X^2)}
{\,2+\frac{Y}{2}
+\cfrac{1/(60X^2)}
{\,6+\frac{Y}{2}
+\cfrac{9/(560X^2)}
{\,12+\frac{Y}{2}+\ddots}}}.
\]
Thus the BOCRS decaying-branch continued fraction furnishes an implicit relation between the
two normalized constants \(C/\pi\) and \(L_\tau(1)/C\). 
\end{remark}

For each integer $N\ge 1$, define the depth-$N$ truncants by
\begin{equation}\label{eq:q-truncants}
q_N^{(N)}:=N(N+1)-\lambda,
\qquad
q_n^{(N)}:=n(n+1)-\lambda+\frac{\beta_n}{q_{n+1}^{(N)}}
\quad(1\le n\le N-1).
\end{equation}
Thus $q_1^{(N)}$ is the finite continued fraction
\begin{equation}\label{eq:q1-truncant}
q_1^{(N)}
=
2-\lambda+
\cfrac{\beta_1}{6-\lambda+
\cfrac{\beta_2}{12-\lambda+
\cfrac{\beta_3}{\ddots+
\cfrac{\beta_{N-1}}{N(N+1)-\lambda}}}}.
\end{equation}
Write
\begin{equation}\label{eq:K-def}
q_n^{(N)}=\frac{K_n^{(N)}}{K_{n+1}^{(N)}},
\end{equation}
where
\begin{equation}\label{eq:K-recurrence}
K_{N+1}^{(N)}=1,
\qquad
K_N^{(N)}=N(N+1)-\lambda,
\end{equation}
and
\begin{equation}\label{eq:K-recursive}
K_n^{(N)}=(n(n+1)-\lambda)K_{n+1}^{(N)}+\beta_n K_{n+2}^{(N)}
\qquad(1\le n\le N-1).
\end{equation}

Accordingly, whenever we use \(q_n^{(N)}\) algebraically, we regard it as the rational
function
\[
q_n^{(N)}=\frac{K_n^{(N)}}{K_{n+1}^{(N)}}\in \Q(a,\lambda).
\]
Its continued-fraction interpretation at a point \((a,\lambda)\) requires the successive
denominators to be nonzero.

\begin{proposition}\label{prop:truncants-determinants}
Define
\[
T_0(a,\lambda):=(-\lambda),
\]
and, for \(N\ge 1\),
\begin{equation}\label{eq:TN-def}
T_N(a,\lambda):=
\begin{pmatrix}
-\lambda & \dfrac{a}{3} & 0 & \cdots & 0 \\
-a & 2-\lambda & \dfrac{2a}{5} & \ddots & \vdots \\
0 & -\dfrac{2a}{3} & 6-\lambda & \ddots & 0 \\
\vdots & \ddots & \ddots & \ddots & \dfrac{Na}{2N+1} \\
0 & \cdots & 0 & -\dfrac{Na}{2N-1} & N(N+1)-\lambda
\end{pmatrix}.
\end{equation}
For \(N\ge 1\) and \(0\le n\le N+1\), let \(\Delta_n^{(N)}\) denote the determinant of the
trailing principal submatrix of \(T_N(a,\lambda)\) obtained by deleting the first \(n\) rows
and columns; in particular,
\[
\Delta_0^{(N)}=\det T_N(a,\lambda),
\qquad
\Delta_{N+1}^{(N)}:=1.
\]
Then the following hold.
\begin{enumerate}[label=\textup{(\roman*)}]
\item
One has
\[
\Delta_n^{(N)}=K_n^{(N)}
\qquad(1\le n\le N+1).
\]
\item
One has
\begin{equation}\label{eq:detT-K}
\det T_N(a,\lambda)=-\lambda K_1^{(N)}+\frac{a^2}{3}K_2^{(N)}.
\end{equation}
Equivalently,
\begin{equation}\label{eq:truncation-rational-identity}
\lambda q_1^{(N)}-\frac{a^2}{3}
=
-\frac{\det T_N(a,\lambda)}{K_2^{(N)}}
\end{equation}
as an identity in \(\Q(a,\lambda)\). In particular, on the locus \(K_2^{(N)}\neq 0\), the
equation
\begin{equation}\label{eq:truncation-eqn}
\lambda q_1^{(N)}=\frac{a^2}{3}
\end{equation}
is equivalent to
\begin{equation}\label{eq:truncation-det}
\det T_N(a,\lambda)=0.
\end{equation}
\end{enumerate}
\end{proposition}

\begin{proof}
For \(0\le n\le N-1\), first-row expansion of the trailing principal submatrix starting at
index \(n\) gives
\[
\Delta_n^{(N)}
=
(n(n+1)-\lambda)\Delta_{n+1}^{(N)}
-\frac{a(n+1)}{2n+3}\left(-\frac{a(n+1)}{2n+1}\right)\Delta_{n+2}^{(N)}.
\]
Thus
\[
\Delta_n^{(N)}
=
(n(n+1)-\lambda)\Delta_{n+1}^{(N)}
+\beta_n\Delta_{n+2}^{(N)}
\qquad(0\le n\le N-1),
\]
with terminal values
\[
\Delta_N^{(N)}=N(N+1)-\lambda,
\qquad
\Delta_{N+1}^{(N)}=1.
\]
Comparing this with \eqref{eq:K-recurrence}--\eqref{eq:K-recursive}, we obtain
\[
\Delta_n^{(N)}=K_n^{(N)}
\qquad(1\le n\le N+1),
\]
which proves \textup{(i)}.

Taking \(n=0\) in the same expansion gives
\[
\det T_N(a,\lambda)
=
-\lambda\Delta_1^{(N)}+\beta_0\Delta_2^{(N)}
=
-\lambda K_1^{(N)}+\frac{a^2}{3}K_2^{(N)},
\]
which is \eqref{eq:detT-K}. Since \(q_1^{(N)}=K_1^{(N)}/K_2^{(N)}\) in \(\Q(a,\lambda)\),
\[
\lambda q_1^{(N)}-\frac{a^2}{3}
=
\frac{\lambda K_1^{(N)}-\frac{a^2}{3}K_2^{(N)}}{K_2^{(N)}}
=
-\frac{\det T_N(a,\lambda)}{K_2^{(N)}},
\]
which is \eqref{eq:truncation-rational-identity}. The pointwise equivalence of
\eqref{eq:truncation-eqn} and \eqref{eq:truncation-det} on the locus \(K_2^{(N)}\neq 0\)
is now immediate. This proves \textup{(ii)}.

\end{proof}

\begin{remark}\label{rem:truncation-domain}
The condition \(K_2^{(N)}\neq 0\) in Proposition~\ref{prop:truncants-determinants}\textup{(ii)}
is necessary. For example, when \(N=3\), \(\lambda=0\), and \(a=\pm 2i\sqrt{70}\), one has
\[
K_2^{(3)}=(6-\lambda)(12-\lambda)+\frac{9a^2}{35}=0,
\qquad
K_1^{(3)}=-896\neq 0,
\]
so \(q_1^{(3)}=K_1^{(3)}/K_2^{(3)}\) is undefined. Nevertheless,
\[
\det T_3(a,\lambda)
=
-\lambda K_1^{(3)}+\frac{a^2}{3}K_2^{(3)}
=
0.
\]
Thus the determinantal equation may hold at points where the continued-fraction truncation
equation is not defined.
\end{remark}

\begin{theorem}[Exact finite-compression identity]\label{thm:truncation-identity}
For every \(N\ge 0\),
\[
\widehat u_{N+1}(a,\lambda)=C_{N+1}\det T_N(a,\lambda).
\]
Consequently, for every \(N\ge 0\),
\[
\widehat u_{N+1}(a,\lambda)=0
\]
if and only if
\[
\det T_N(a,\lambda)=0.
\]
If \(N\ge 1\) and \(K_2^{(N)}(a,\lambda)\neq 0\), then these equations are also equivalent to
\[
\lambda q_1^{(N)}=\frac{a^2}{3}.
\]
\end{theorem}

\begin{proof}
Set
\[
E_{-1}:=1,
\qquad
E_N:=\det T_N(a,\lambda)
\qquad(N\ge 0).
\]
Expanding along the last row of \(T_N(a,\lambda)\), we get
\[
E_N=(N(N+1)-\lambda)E_{N-1}
+\frac{N^2a^2}{(2N-1)(2N+1)}E_{N-2}
\qquad(N\ge 1),
\]
with
\[
E_0=-\lambda.
\]
Now define
\[
U_n:=C_nE_{n-1}
\qquad(n\ge 0).
\]
Then \(U_0=1\) and \(U_1=-2\lambda\). Moreover, for \(n\ge 1\),
\begin{align*}
U_{n+1}
&=C_{n+1}E_n \\
&=\frac{C_{n+1}}{C_n}\bigl(n(n+1)-\lambda\bigr)U_n
+\frac{C_{n+1}}{C_{n-1}}\frac{n^2a^2}{(2n-1)(2n+1)}U_{n-1} \\
&=\frac{2(2n+1)}{n+1}\bigl(n(n+1)-\lambda\bigr)U_n
+\frac{4n}{n+1}a^2U_{n-1} \\
&=\frac{4n+2}{n+1}\bigl(n(n+1)-\lambda\bigr)U_n
+\frac{4n}{n+1}a^2U_{n-1}.
\end{align*}
Hence \((U_n)_{n\ge 0}\) satisfies the same recurrence and the same initial conditions as
\((\widehat u_n(a,\lambda))_{n\ge 0}\). Therefore
\[
U_n=\widehat u_n(a,\lambda)
\qquad(n\ge 0).
\]
Therefore
\[
U_n=\widehat u_n(a,\lambda)
\qquad(n\ge 0),
\]
so
\[
\widehat u_{N+1}(a,\lambda)=C_{N+1}\det T_N(a,\lambda)
\qquad(N\ge 0).
\]
Since \(C_{N+1}\neq 0\), this immediately yields
\[
\widehat u_{N+1}(a,\lambda)=0
\quad\Longleftrightarrow\quad
\det T_N(a,\lambda)=0.
\]
If \(N\ge 1\) and \(K_2^{(N)}(a,\lambda)\neq 0\), then
Proposition~\ref{prop:truncants-determinants} shows that these are further equivalent to
\[
\lambda q_1^{(N)}=\frac{a^2}{3}.
\]
This completes the proof of Theorem~\ref{thm:truncation-identity}.
\end{proof}

We now turn from exact integrality to a first arithmetic consequence. The point is that
Theorem~\ref{thm:main}, together with the elementary degree bounds
\eqref{eq:uhat-degree-bounds}, produces integral multiples of the coefficients
\(\widehat u_n(a,\lambda)\) whenever \(a^2\) and \(\lambda\) are rational. If, in addition,
the associated even generating series has exponential type at most \(2a\), then Cauchy's
estimate forces those integral multiples to decay superexponentially. The only way an integer
sequence can behave in this fashion is that it eventually vanishes. This yields the following
rational-slice obstruction.

\begin{theorem}\label{thm:rational-slice}
Let $a>0$ and $\lambda\in\mathbb C$. Assume that the power series
\begin{equation}\label{eq:F-def-rational-slice}
F(z):=\sum_{n=0}^{\infty}\widehat u_n(a,\lambda)z^{2n}
\end{equation}
defines an entire nonpolynomial function of exponential type at most $2a$. Then $a^2$ and $\lambda$ are not both rational.
\end{theorem}

\begin{proof}
Suppose, to the contrary, that
\[
a^2=\frac{r}{s}\in\Q,
\qquad
\lambda=\frac{p}{q}\in\Q,
\]
with $r,s,p,q\in\Z$ and $s,q\ge 1$. By \eqref{eq:uhat-integrality} and \eqref{eq:uhat-degree-bounds}, we have
\[
(sq)^n\widehat u_n(a,\lambda)\in\Z
\qquad(n\ge 0).
\]
Fix $\varepsilon\in(0,a)$. Since $F$ is entire of exponential type at most $2a$, there exists a constant $A_{\varepsilon}>0$ such that
\[
|F(z)|\le A_{\varepsilon}e^{2(a+\varepsilon)|z|}
\qquad(z\in\mathbb C).
\]
By Cauchy's estimate, for every $R>0$,
\[
|\widehat u_n(a,\lambda)|
\le \frac{\sup_{|z|=R}|F(z)|}{R^{2n}}
\le A_{\varepsilon}\frac{e^{2(a+\varepsilon)R}}{R^{2n}}.
\]
Choosing
\[
R=\frac{n}{a+\varepsilon}
\]
gives
\[
|\widehat u_n(a,\lambda)|
\le A_{\varepsilon}\left(\frac{e(a+\varepsilon)}{n}\right)^{2n}.
\]
Therefore
\[
(sq)^n|\widehat u_n(a,\lambda)|
\le A_{\varepsilon}\left(\frac{sq\,e^2(a+\varepsilon)^2}{n^2}\right)^n
\longrightarrow 0.
\]
Since $(sq)^n\widehat u_n(a,\lambda)$ is an integer for every $n$, it follows that
\[
(sq)^n\widehat u_n(a,\lambda)=0
\]
for all sufficiently large $n$. Hence $\widehat u_n(a,\lambda)=0$ for all sufficiently large $n$, and so the power series \eqref{eq:F-def-rational-slice} is a polynomial, contrary to assumption.
\end{proof}

We now apply Theorem~\ref{thm:rational-slice} to the BOCRS decaying branch. Since BOCRS
formulate the coefficient recurrence in the \(L_{1,b}\)-normalization, whereas the decaying
branch is naturally written in the \(L_{a,1}\)-normalization, we first pass to the rescaled
eigenfunction \(g(z)=f(az)\). This places the even product in exactly the framework of
Theorem~\ref{thm:rational-slice}.

\begin{corollary}\label{cor:BOCRS-rational-slice}
Let \(f\) be the normalized eigenfunction corresponding to the decaying BOCRS branch
with parameter \(a>0\) and eigenvalue \(\lambda\). Assume that \(f(z)f(-z)\) is not a
polynomial. Then \(a^2\) and \(\lambda\) are not both rational. In particular, if \(a^2\in\Q\),
then \(\lambda\notin\Q\).
\end{corollary}

\begin{proof}
By definition of the decaying BOCRS branch, \(f\) is the normalized \(\lambda\)-eigenfunction
of \(L_{a,1}\). Define
\[
g(z):=f(az).
\]
If \(R_kf(z):=f(kz)\), then the BOCRS scaling law \cite[\S4]{BOCRS} gives
\[
L_{1,a}g=L_{1,a}R_af=R_aL_{a,1}f=\lambda g.
\]
Also \(g(0)=f(0)=1\), so \(g\) is a normalized eigenfunction of \(L_{1,a}\).

Therefore, by \cite[\S10.3]{BOCRS},
\[
g(z)g(-z)=\sum_{n=0}^{\infty}\widehat u_n(a,\lambda)z^{2n}.
\]
By \cite[Corollary~7.2]{BOCRS}, the function \(g\) is entire of exponential type at most
\(a\). Hence \(g(z)g(-z)\) is entire of exponential type at most \(2a\).

Since
\[
g(z)g(-z)=f(az)f(-az),
\]
the function \(g(z)g(-z)\) is a polynomial if and only if \(f(z)f(-z)\) is a polynomial.
By assumption, it is not a polynomial. Theorem~\ref{thm:rational-slice} applied to
\(g(z)g(-z)\) now shows that \(a^2\) and \(\lambda\) are not both rational. The final
assertion is immediate.
\end{proof}

We next specialize Corollary~\ref{cor:BOCRS-rational-slice} to the
H\"ormander--Bernhardsson branch. In the notation of \cite{BOCRS}, \(C\) denotes the sharp
point-evaluation constant and
\[
\varphi(z)=\Phi(z)\Phi(-z).
\]
The BOCRS rescaling identifies the corresponding normalized spectral pair as
\[
a_*=\frac{\pi}{4C},
\qquad
\lambda_*=-\frac{L_\tau(1)}{2C},
\]
so the rational-slice obstruction becomes a concrete arithmetic statement about the two
distinguished constants \(C\) and \(L_\tau(1)\).

\begin{corollary}\label{cor:HB-normalized-obstruction}
For the Hörmander--Bernhardsson branch, the two numbers
\[
\left(\frac{\pi}{4C}\right)^2
\qquad\text{and}\qquad
-\frac{L_{\tau}(1)}{2C}
\]
are not both rational. In particular, at least one of $C/\pi$ and $L_{\tau}(1)/C$ is irrational.
\end{corollary}

\begin{proof}
By \cite[Theorem~1.1]{BOCRS}, the function \(\Phi\) is the normalized eigenfunction of
\[
L_{1/(2C),\,\pi/2}
\]
with eigenvalue
\[
\lambda_*:=-\frac{L_{\tau}(1)}{2C}.
\]
Set
\[
a_*:=\frac{\pi}{4C},
\qquad
h(z):=\Phi\!\left(\frac{z}{2C}\right).
\]
If \(R_kf(z):=f(kz)\), then \cite[\S4]{BOCRS} gives the scaling law
\[
L_{a,b}R_k=R_kL_{ka,k^{-1}b}.
\]
Applying this with \(a=1\), \(b=a_*\), and \(k=\frac{1}{2C}\), we obtain
\[
L_{1,a_*}h
=
L_{1,a_*}R_{1/(2C)}\Phi
=
R_{1/(2C)}L_{1/(2C),\,\pi/2}\Phi
=
\lambda_* h.
\]
Moreover \(h(0)=\Phi(0)=1\). Hence \(h\) is a normalized eigenfunction of \(L_{1,a_*}\), and
therefore
\[
h(z)h(-z)=\sum_{n=0}^{\infty}\widehat u_n(a_*,\lambda_*)z^{2n}.
\]

Since
\[
h(z)h(-z)=\varphi\!\left(\frac{z}{2C}\right),
\]
and \(\varphi\) has infinitely many zeros \(\pm\tau_n\), the function \(h(z)h(-z)\) is not a
polynomial. By \cite[Corollary~7.2]{BOCRS}, \(h\) is entire of exponential type at most
\(a_*\), so \(h(z)h(-z)\) is entire of exponential type at most \(2a_*\). Therefore
Theorem~\ref{thm:rational-slice} implies that \(a_*^2\) and \(\lambda_*\) are not both rational.

If both \(C/\pi\) and \(L_{\tau}(1)/C\) were rational, then
\[
a_*=\frac{\pi}{4C}=\frac{1}{4(C/\pi)}\in\Q
\qquad\text{and}\qquad
\lambda_*=-\frac{1}{2}\frac{L_{\tau}(1)}{C}\in\Q,
\]
which is impossible. Therefore at least one of \(C/\pi\) and \(L_{\tau}(1)/C\) is irrational.
\end{proof}

\begin{remark}\label{rem:not-Ltau}
Corollary~\ref{cor:HB-normalized-obstruction} does not imply that \(L_{\tau}(1)\) itself is
irrational.  Indeed, for the H\"ormander--Bernhardsson branch the BOCRS rescaling gives
\[
a_*=\frac{\pi}{4C},
\qquad
\lambda_*=-\frac{L_{\tau}(1)}{2C}.
\]
Thus a hypothesis such as \(L_{\tau}(1)\in\Q\) places the spectral pair only on the line
\[
\lambda=-\frac{2L_{\tau}(1)}{\pi}\,a,
\]
and does \emph{not} imply that either \(a_*^2\) or \(\lambda_*\) is rational.  The obstruction
from Corollary~\ref{cor:HB-normalized-obstruction} therefore controls only the normalized pair
\[
\left(\frac{\pi}{4C}\right)^2,
\qquad
-\frac{L_{\tau}(1)}{2C},
\]
not the number \(L_{\tau}(1)\) by itself.
\end{remark}

\begin{remark}\label{rem:Lplusminus2}
Bondarenko--Ortega-Cerd\`a--Radchenko--Seip formulate in \cite[\S10]{BOCRS} two arithmetic questions: Conjecture~1, namely
\[
L_{+}(-2k)=\frac{L_{+}(2k)}{(2\pi iC)^{2k}},
\qquad k\in\Z,
\]
and Conjecture~2, namely the integrality proved in the present paper. The author proved Conjecture~1 in \cite[Theorem~1.2]{NguyenDangHB}. Combining his theorem with the BOCRS identities
\[
L_{+}(2)=-4C\,L_{\tau}(1),
\qquad
\pi^2 a_1=-\frac{L_{\tau}(1)}{C},
\]
here \(a_1\) denotes the coefficient \(a_1\) in the BOCRS odd inverse-power expansion of
the positive zeros \((\tau_n)\); see \cite[(1.7)]{BOCRS}. One obtains
\[
L_{+}(-2)
=
\frac{L_{+}(2)}{(2\pi iC)^2}
=
\frac{L_{\tau}(1)}{\pi^2C}
=
-a_1.
\]
\end{remark}

\section{Weighted arithmetic and the coefficient ring of \texorpdfstring{$\varphi$}{phi}}
\label{sec:weighted-arithmetic}

The integrality theorem admits a sharper refinement: not only are the coefficients integral,
their monomial support is confined to a natural weighted region. This reflects the structure of
the recurrence itself: multiplication by \(\lambda\) raises the relevant weight by \(1\),
whereas multiplication by \(x\) raises it by \(2\) but comes from the lower-order term
\(\widetilde u_{n-1}\). The resulting weighted filtration gives the following precise support
bound.

\begin{proposition}\label{prop:weighted-support}
Define a weight on monomials by
\[
w(x)=2,
\qquad
w(\lambda)=1,
\qquad
w(x^r\lambda^s)=2r+s.
\]
For a nonzero polynomial
\[
F(x,\lambda)=\sum_{r,s} c_{r,s}x^r\lambda^s
\]
set
\[
\deg_w F:=\max\{2r+s:c_{r,s}\ne 0\},
\]
and \(\deg_w 0:=-\infty\). Then for every \(n\ge 0\),
\[
\deg_w\widetilde u_n\le n.
\]
Equivalently, every monomial \(x^r\lambda^s\) occurring in \(\widetilde u_n(x,\lambda)\)
satisfies
\[
2r+s\le n.
\]
\end{proposition}

\begin{proof}
We argue by induction on \(n\). For \(n=0\), we have \(\widetilde u_0=1\), so
\(\deg_w\widetilde u_0=0\). For \(n=1\), the recurrence \eqref{eq:utilde-recurrence} at \(n=0\)
gives \(\widetilde u_1=-2\lambda\), so \(\deg_w\widetilde u_1=1\).

Now assume
\[
\deg_w\widetilde u_n\le n,
\qquad
\deg_w\widetilde u_{n-1}\le n-1
\]
for some \(n\ge 1\). From \eqref{eq:utilde-recurrence},
\[
\widetilde u_{n+1}
=
\frac{4n+2}{n+1}\bigl(n(n+1)-\lambda\bigr)\widetilde u_n
+\frac{4n}{n+1}x\,\widetilde u_{n-1}.
\]
Multiplication by the constant \(n(n+1)\) does not change the weight, multiplication by
\(\lambda\) increases the weight by \(1\), and multiplication by \(x\) increases the weight by \(2\).
Therefore
\[
\deg_w\bigl((n(n+1)-\lambda)\widetilde u_n\bigr)\le n+1,
\qquad
\deg_w\bigl(x\,\widetilde u_{n-1}\bigr)\le (n-1)+2=n+1.
\]
Hence \(\deg_w\widetilde u_{n+1}\le n+1\), completing the induction.
\end{proof}

This weighted support bound has a concrete consequence for the Taylor coefficients of the
H\"ormander--Bernhardsson extremal function. After inserting the normalized parameters
\[
a_*^2=\frac{\pi^2}{16C^2},
\qquad
\lambda_*=-\frac{L_\tau(1)}{2C},
\]
the inequality \(2r+s\le n\) is exactly what rules out odd powers of \(\pi\) and forces an
additional factor of \(C^n\). In this way, the abstract weight bound becomes an explicit
coefficient-ring statement for \(\varphi\).

\begin{corollary}\label{cor:HB-coefficient-ring}
Let
\[
a_*:=\frac{\pi}{4C},
\qquad
\lambda_*:=-\frac{L_\tau(1)}{2C},
\]
and write
\[
\varphi(z)=\sum_{n=0}^{\infty} c_n z^{2n}.
\]
Then
\begin{equation}\label{eq:cn-uhat}
c_n=(2C)^{2n}\widehat u_n(a_*,\lambda_*)
\qquad(n\ge 0).
\end{equation}
Consequently,
\[
c_n\in \Z[\pi^2,C,L_\tau(1)]
\qquad(n\ge 0).
\]
In fact, one has the stronger divisibility statement
\[
c_n\in C^n\,\Z[\pi^2,C,L_\tau(1)]
\qquad(n\ge 0).
\]
\end{corollary}

\begin{proof}
As in the proof of Corollary~\ref{cor:HB-normalized-obstruction}, set
\[
h(z):=\Phi\!\left(\frac{z}{2C}\right).
\]
Then \(h\) is a normalized eigenfunction of \(L_{1,a_*}\) with eigenvalue \(\lambda_*\), where
\[
a_*=\frac{\pi}{4C},
\qquad
\lambda_*=-\frac{L_\tau(1)}{2C}.
\]
Hence, by \cite[\S10.3]{BOCRS},
\[
h(z)h(-z)=\sum_{n=0}^{\infty}\widehat u_n(a_*,\lambda_*)z^{2n}.
\]
Since
\[
h(z)h(-z)=\varphi\!\left(\frac{z}{2C}\right)
=
\sum_{n=0}^{\infty}c_n\left(\frac{z}{2C}\right)^{2n},
\]
comparing coefficients gives
\[
c_n=(2C)^{2n}\widehat u_n(a_*,\lambda_*)
\qquad(n\ge0).
\]

Now write
\[
\widetilde u_n(x,\lambda)=\sum_{r,s} a_{r,s}x^r\lambda^s
\qquad(a_{r,s}\in \Z).
\]
By Proposition~\ref{prop:weighted-support}, only pairs \((r,s)\) with \(2r+s\le n\) occur.
Using
\[
a_*^2=\frac{\pi^2}{16C^2},
\qquad
\lambda_*=-\frac{L_\tau(1)}{2C},
\]
a monomial \(x^r\lambda^s\) contributes to
\(c_n=(2C)^{2n}\widetilde u_n(a_*^2,\lambda_*)\)
the term
\[
(-1)^s\,2^{\,2n-4r-s}\,
\pi^{2r}\,
C^{\,2n-2r-s}\,
L_\tau(1)^s.
\]
Because \(2r+s\le n\), we have
\[
2n-4r-s=(n-2r-s)+(n-2r)\ge0,
\qquad
2n-2r-s=n+(n-2r-s)\ge n.
\]
Hence every such term belongs to
\[
C^n\Z[\pi^2,C,L_\tau(1)].
\]
Summing over all \((r,s)\), we obtain
\[
c_n\in C^n\,\Z[\pi^2,C,L_\tau(1)]
\subset \Z[\pi^2,C,L_\tau(1)].
\qedhere
\]
\end{proof}

\begin{remark}\label{rem:HB-pi2-sharpening}
BOCRS observe in \cite[\S10.3]{BOCRS} that the Taylor coefficients of \(\varphi\) appear to
lie in \(\Z[\pi,C,L_\tau(1)]\). Corollary \ref{cor:HB-coefficient-ring} sharpens this:
only even powers of \(\pi\) can occur.
\end{remark}

\end{document}